\documentclass[12pt]{article}
\usepackage{amsmath,amssymb}
\usepackage{amsthm}

\def\C{\mathbf{C}}
\def\E{\mathbf{E}}

\def\R{\mathbf{R}}

\def\1{\mathbf{1}}

\def\tr{\rm{tr}}

\def\sgn{\rm{sgn}}

\def\al{\alpha}

\def\pa{\partial}

\def\de{\delta}

\newtheorem{prop}{Proposition}[section]
\newtheorem{theorem}{Theorem}[section]

\newtheorem{remark}{Remark}

\newcommand{\la}{\lambda}
\newcommand{\si}{\sigma}

\newcommand{\om}{\omega}
\newcommand{\Ga}{\Gamma}
\newcommand{\De}{\Delta}

\begin{document}
\title{Dynamic quantum games}
%\thanks{}

\author{Vassili N. Kolokoltsov\thanks{Department of Statistics, University of Warwick,
 Coventry CV4 7AL UK, associate member of HSE, Moscow,
  Email: v.kolokoltsov@warwick.ac.uk}}
\maketitle
	
\begin{abstract}
Quantum games represent the really 21st century branch of game theory, tightly linked to the modern
development of quantum computing and quantum technologies. The main accent in these developments so far was made on
stationary or repeated games. In this paper we aim at initiating the truly dynamic theory with strategies chosen
 by players in real time. Since direct continuous observations are known to destroy quantum evolutions
 (so-called quantum Zeno paradox) the necessary new ingredient for quantum dynamic games must be the theory of
 non-direct observations and the corresponding quantum filtering. Apart from the technical problems in
 organising feedback quantum control in real time, the difficulty in applying this theory for obtaining
 mathematically amenable control systems is due partially to the fact that it leads usually to rather nontrivial
 jump-type Markov processes and/or degenerate diffusions on manifolds, for which the corresponding control
 is very difficult to handle. The starting point for the present research
  is the remarkable discovery (quite unexpected, at least to the author) that there exists a very natural class of homodyne
  detections such that the diffusion processes on projective spaces resulting by filtering under such arrangements
  coincide exactly with the standard Brownian motions (BM) on these spaces. In some cases one can even reduce the
  process to the plain BM on Euclidean spaces or tori. The theory of such motions is well studied
  making it possible to develop a tractable theory of related control and games, which can be
  at the same time practically implemented on quantum optical devices.
\end{abstract}


\section{Introduction}

Quantum games represent the really 21st century branch of game theory, tightly linked to the modern
development of quantum computing and quantum technologies. Initiated by Meyer \cite{MeyerD99},
Eisert,  Wilkens and Lewenstein \cite{EWL99}, and Marinatto and Weber \cite{MW00}, the theory now
boasts of many beautiful results obtained by various authors in numerous publications, see e.g.
surveys \cite{KhanReview18}, \cite{GuoZhang08}, and a mathematically oriented survey \cite{KolQSurv}.
 However, the main accent in these developments was made on
stationary or repeated games. In this paper we aim at initiating the truly dynamic theory with strategies chosen
 by players in real time. Since direct continuous observations are known to destroy quantum evolutions
 (so-called quantum Zeno paradox) the necessary new ingredient for quantum dynamic games must be the theory of
 non-direct observations and the corresponding quantum filtering. This theory was essentially
 developed by Belavkin in the 80s of the last century, in \cite{Bel87}, \cite{Bel88}, \cite{Bel92},
 see \cite{BoutHanJamQuantFilt} for a readable modern account. There is an important work under way on
 the technical side of organising feedback quantum control in real time, see e.g. \cite{Armen02Adaptive},
 \cite{Bushev06Adaptive} and \cite{WiMilburnBook}. The difficulty in applying this theory for obtaining
 mathematically amenable control systems is due partially to the fact that it leads usually to rather nontrivial
 jump-type Markov processes and/or degenerate diffusions on manifolds, for which the corresponding control
 (an even more so games) is very difficult to handle.

 The starting point for the present research
  was the remarkable discovery (quite unexpected, at least to the author) that there exists a very natural class of homodyne
  detections such that the diffusion processes on spheres or projective spaces resulting by filtering under such arrangements
  coincide exactly with the standard Brownian motions (BM) on these Riemannian manifolds, that is, the processes
   generated by the invariant Laplace -Beltrami operator. For qubits the basic example of such
  special arrangements is the choice of the three Pauli matrices as the coupling operators governing the interaction
  with the optical measuring devices. For qudits the corresponding matrices can be chosen as the generalized Pauli
 or Gell-Mann matrices. Another unexpected feature of these special arrangements is that the corresponding diffusions
 written with respect to the output process coincide exactly with the diffusions written with respect to the so-called
 innovation process that plays the key role in the theory of quantum feedback control.
 The theory of the BM on compact Riemannian manifolds is well studied in stochastic analysis and operator theory on manifolds,
 making it possible to develop a tractable theory of related control and games, which can be
  at the same time practically implemented on quantum optical devices. This theory is based on the ability to build classical or mild solutions to the corresponding Hamilton-Jacobi-Bellman-Isaacs (HJB-Isaacs) equations on compact Riemannian manifolds,
  which makes it more elementary than the approach to stochastic control based on the viscosity solutions,
  for which we refer to \cite{Zhu14} and references therein.
  Moreover,  in some cases (essentially when all controlled Hamiltonian operators commute)
  the filtered dynamics turns out to be govern by the standard Brownian motion
  on Euclidean spaces and tori, that is by the diffusion processes generated by the standard plain Laplacian
  in $\R^d$ or a torus.

  The content of the paper is as follows.
  In the next section we briefly explain the necessary tools from the theory of continuous quantum measurement and filtering.
  In the following two sections we introduce our main homodyne detection schemes (first for qubits and then for qudits)
   that allow one to turn the problems of dynamic quantum filtering, control and games into the problems of the drift control
   of the standard Brownian motions on the complex projective spaces. In Sections \ref{secgameon Riemann} and  \ref{secqugameth}
we build the theory of classical and mild solutions of the HJB-Isaacs equations on Riemannian manifolds leading to the
theory of dynamic control and games on compact Riemannian manifolds and thus automatically to the theory of
 quantum dynamic control and games under the special homodyne detection schemes.
In Section \ref{secqugametotori} we introduce yet another homodyne detection scheme that leads to the simpler drift controls on tori,
which works however only in case when all controlled Hamiltonian operators commute. In Section \ref{secqugametotoriex} an exactly
solvable model is presented, which can be considered as a kind of dynamic extension of the initial quantum coin flipping game of Meyer.
In Section \ref{secqugametoEuc} yet another detection scheme is developed that turns the problem of quantum dynamic control to
 the drift control of the standard BM in Euclidean spaces.
In Sections \ref{sectwoatoms} and \ref{sectwoatomsnonzerosu} a version of the theory is developed for players acting on different atoms,
thus for the dynamic games set in the spirit of papers \cite{EWL99} and \cite{MW00}.
Some conclusions and perspectives are drawn and open questions posed in Section \ref{conc}.

\section{Prerequisites: nondemolition observation and quantum filtering}

The general theory of quantum non-demolition observation, filtering and resulting
feedback control was built essentially in papers  \cite{Bel87}, \cite{Bel88}, \cite{Bel92}.
A very good readable introduction is given in \cite{BoutHanJamQuantFilt}.
 We shall describe briefly the main result of this theory.

The non-demolition measurement of quantum systems can be organised in two versions:
photon counting and homodyne detection. One of the first mathematical results on the control
with photon counting measurement was given in \cite{Kol92}, which can be used to develop the
corresponding  game theoretical version. But here we fully concentrate on the homodyne
(mathematically speaking, diffusive type) detection. Under this type of measurement the
output process $Y_t$ is a usual Brownian motion (under appropriate probability distribution).
There are several (by now standard) ways of writing down the quantum filtering equation for
states resulting from the outcome of such process. The one which is the most convenient to our
purposes is the following linear Belavkin filtering equation (which is a particular version
of the stochastic Schr\"odinger equation)
describing the a posteriori (pure but not normalized) state:

\begin{equation}
\label{eqqufiBlin}
d\chi =-[iH\chi +\frac12 L^*L \chi ]\,dt+L\chi dY_t,
\end{equation}
where the unknown vector $\chi$ is from the Hilbert space of the observed quantum system,
which we shall loosely referred to everywhere as the atom,
the self-adjoint operator $H$ is the Hamiltonian
of the corresponding initial (non-observed) quantum evolution and the operator $L$
is the coupling operator of the atom to the optical measurement device specifying the chosen version
of the homodyne detection. Very often the operator $L$ is
chosen to be self-adjoint, in which case equation \eqref{eqqufiBlin} reduces to the simper form
\begin{equation}
\label{eqqufiBlins}
d\chi =-[iH\chi +\frac12 L^2 \chi ]\,dt+L\chi dY_t.
\end{equation}

An important part in the theory is played by the so-called innovation process
\begin{equation}
\label{eqdefinnov}
dB_t=dY_t-\langle L+L^* \rangle_{\chi} \, dt,
\end{equation}
where for an operator $A$ and a vector $v$ in a Hilbert space we use the (more or less standard)
notation for the average value of $A$ in $v$:
\[
\langle A \rangle_v=\frac{(v,Av)}{(v,v)}.
\]

The innovation process is in some sense a more natural driving noise to deal with,
because it turns out to be the standard Brownian motion (or the Wiener process)
with respect to the fixed (initial vacuum) state of the homodyne detector, while
the output process $Y_t$ is a Brownian motion with respect to the states transformed
by the (quite complicated) interaction of the quantum system and optical device, which can also be obtained
by the Girsanov transformation from the innovation process $B_t$.
Therefore another well used version of equation \eqref{eqqufiBlin} is the nonlinear equation
on the normalized vector $\phi=\chi/|\chi|$, which can be obtained directly from \eqref{eqqufiBlin}
by the classical Ito formula (using the classical Ito rule for the differentials of the Wiener processes,
$dY_t dY_t=dt$), but written in terms of the innovation process $B_t$.

The theory extends naturally to the case of several, say $N$, coupling operators $\{L_j\}$,
where the quantum filtering is described by the following direct extension
of equation \eqref{eqqufiBlin}:
 \begin{equation}
\label{eqqufiBlinn}
d\chi =-[iH\chi +\frac12 \sum_j L_j^*L_j \chi ]\,dt+\sum_j L_j\chi dY^j_t,
\end{equation}
with the $N$-dimensional output process $Y_t=\{Y^j_t\}$.
The corresponding innovation process is the standard $N$-dimensional
Wiener process with the coordinate differentials
\[
dW^j_t=dY^j_t-\langle L_j+L_j^* \rangle_{\chi} \, dt.
\]

The theory of quantum filtering reduces the analysis of quantum dynamic control and games
to the controlled version of evolutions \eqref{eqqufiBlinn}. The simplest situation concerns the case
 when the homodyne device is fixed, that is the operators $L_j$ are fixed, and the players can control the
 Hamiltonian $H$, say, by applying appropriate electric or magnetic fields to the atom. Thus equation
 \eqref{eqqufiBlinn} becomes modified by allowing $H$ to depend on one or several control parameters.
 One can even prove a rigorous mathematical result, the so-called separation principle
 (see \cite{BoutHanQuantumSepar}), that shows that the effective control of an observed quantum system
 (that can be based in principle on the whole history of the interaction of the atom and optical devices)
 can be reduced to the Markovian feedback control of the quantum filtering equation, with the feedback
 at each moment depending only on the current (filtered) state of the atom.

\begin{remark} The filtering equation \eqref{eqqufiBlin} was initially derived from the interaction
of the atom and optic devices described by
the unitary evolution solving the quantum stochastic equation
\[
dU_t=(L \, dA_t^*-L^* \, dA_t -\frac12 L^*L \, dt -i H \, dt) U_t,
\]
where $A_t, A_t^*$ are the Hudson-Partasarathy differentials of the quantum stochastic Wiener noise
(built from the annihilation and creation operators).
It can be shown (see e.g. \cite{AFL}) that this evolution represents the Markovian approximation
to the more realistic quantum dynamics
\[
\dot U_t =[-iH +L a^*(t,0)-L^* a(t,0)] U_t,
\]
driven by a stationary Gaussian wide-band noise of the annihilation operators $a(t,r)$.
 More elementary derivations of the main filtering equation (bypassing heavy use
 of quantum stochastic calculus) are also available. It can be obtained
from an appropriate limit of sequential discrete observation scheme, see e.g. \cite{BelKol} or
\cite{Pellegrini}. A derivation from the theory of instruments was given in \cite{BarchBel}
 and \cite{Holevo91}.
\end{remark}

\section{Special homodyne detection leading to the Laplace-Beltrami operator on a sphere (for qubits)}
\label{qugametosphere}

For a qubit the Hilbert space of an atom is $\C^2$. Since the pure state of a quantum
 system is specified by a vector in the Hilbert space up to a multiplier, the actual state space
 is the one-dimensional complex projective space or a two-dimensional sphere, often referred to as the Bloch sphere.
Hence as the natural coordinate outside the state vector $(0,1)$ one can take the complex number
 $w=\chi_1/\chi_0$. It is straightforward to rewrite evolution \eqref{eqqufiBlinn} in $\C^2$ in terms of $w$.
 Namely, from the equation for the first coordinate $\chi_0$ (and Ito's rule for the function $1/x$)
 we find the equation for $\chi_0^{-1}$:
 \[
 d\chi_0^{-1}
 =\frac{1}{\chi_0^2}[iH\chi +\frac12 \sum_j L_j^*L_j \chi ]_0\,dt-\frac{1}{\chi_0^2}\sum_j (L_j\chi)_0 \, dY^j_t
 +\frac{1}{\chi_0^3} \sum_j (L_j\chi)_0^2 \,dt
 \]
 and then using the Ito product rule for the product $\chi_1\chi_0^{-1}$ we find that
  \[
dw=\frac{w}{\chi_0}[i(H\chi)_0 +\frac12 (\sum_j L_j^*L_j\chi)_0 ]\,dt
-\frac{w}{\chi_0}\sum_j (L_j \chi)_0 \, dY^j_t
+\frac{w}{\chi_0^2}\sum_j(L_j \chi)_0^2 \, dt
\]
\[
-\frac{1}{\chi_0}[i(H\chi)_1+\frac12 (\sum_j L_j^*L_j\chi)_1 ]\,dt
+\frac{1}{\chi_0}\sum_j(L_j \chi)_1 \, dY^j_t
-\frac{1}{\chi_0^2}\sum_j (L_j \chi)_0 (L_j \chi)_1  \ dt,
\]
and finally the {\it quantum filtering equation in terms
of the projective coordinates} $w$:
 \[
dw =i[w(HW)_0-(HW)_1]\, dt+\frac12 [\sum_j w(L_j^*L_jW)_0- (L_j^*L_jW)_1 ] \, dt
\]
 \begin{equation}
\label{eqqufiBlinnind2w}
+\sum_j[w(L_jW)_0^2 -(L_jW)_0 (L_jW)_1]\, dt
+\sum_j[(L_jW)_1 -w(L_jW)_0]\, dY^j_t,
\end{equation}
where, for convenience, we have introduced the vector $W=(1,w)=\chi/\chi_0$.
Equivalently it can be rewritten in terms of the innovation processes expressed in terms of $w$ as
 \begin{equation}
\label{eqinnovationinw}
dB^j_t=dY^j_t-\langle L_j+L_j^* \rangle_W \, dt.
\end{equation}

\begin{remark} Though this is not of use for us, let us mention that coordinates $w$ can be obtained by the
stereographic projection from the Stokes parameters $(x_1,x_2,x_3)$ (or a polar vector)
describing in the most common way the Bloch sphere
of the pure quantum states of a qubit.
\end{remark}

We are interested in choosing $\{L_j\}$ in a way to make the diffusion on the Bloch sphere
defined by equation \eqref{eqqufiBlinnind2w} as simple as possible, at least to make it nondegenerate,
that is, with the second order part of the diffusion operator being elliptic.

The Hamiltonian operator does not enter the noise term, and consequently it does not play role in this question.
If $L$ consists of just one operator, the diffusion is definitely degenerate, though it may be
hypoelliptic (see \cite{Kol95}). If there are two operators $L_j$, one usually gets diffusions
that are elliptic almost everywhere (see Sections \ref{secqugametoEuc} and \ref{conc}). Turning to the case
of three operators $L_j$ it is natural to try the simplest three operators on qubits,
namely the three Pauli operators
\[
\si_1=\left(\begin{aligned}
& 0 \quad 1 \\
& 1 \quad 0
\end{aligned}
\right),
\quad
 \si_2=\left(\begin{aligned}
& 0 \quad -i \\
& i \quad \quad 0
\end{aligned}
\right),
\quad
\si_3=\left(\begin{aligned}
& 1 \quad \quad 0 \\
& 0 \quad -1
\end{aligned}
\right).
\]
Since $\si_j$ are self-adjoint and $\si_j^2=\1$, the second term in equation
in this case is seen directly to vanish. Moreover, explicit calculation of the third term shows that
it vanishes as well, so that the filtering equation \eqref{eqqufiBlinnind2w} simplifies to
\[
dw =i[w(HW)_0-(HW)_1]\, dt
+\sum_j[(\si_jW)_1 -w(\si_jW)_0]\, dY^j_t
\]
\begin{equation}
\label{eqqufiBlinnind2wP}
=i[(h_{00}+h_{01}w)w-(h_{10}+h_{11}w)]\, dt
+(1-w^2) \, dY^1_t+i(1+w^2) \, dY^2_t-2w \, dY^3_t,
\end{equation}
where $h_{jk}$ denote the entries of the  $2 \times 2$-matrix $H$.

With this equation two remarkable effects occur.

\begin{prop}
\label{proponPaulidif}
(i) Writing equation \eqref{eqqufiBlinnind2wP} in terms of the innovation process $dB^j=dY^j-2\langle \si_j \rangle_W \, dt$,
it takes exactly the same form \eqref{eqqufiBlinnind2wP} with $B^j$ instead of $Y^j$ (all new terms with the differential $dt$
cancel).

(ii) The diffusion operator $D$ corresponding to equation \eqref{eqqufiBlinnind2wP} with vanishing $H$
 takes the form
\begin{equation}
\label{eqquaqdrfromPau}
D S(x,y)= \frac12 (1+x^2+y^2)^2\left(\frac{\pa ^2S}{\pa x^2}+\frac{\pa ^2S}{\pa y^2}\right),
\end{equation}
in terms of the real coordinates $x,y$, where $w=x+iy$,
so that $D=2\De_{sp}$, where $\De_{sp}$ is the Laplace-Beltrami operator
on the 2-dimensional sphere written in stereographic coordinates.
\end{prop}

\begin{proof}
This is done by direct inspection. For instance, to prove (ii), we can write the equation \eqref{eqqufiBlinnind2wP}
with vanishing $H$ in terms of the real and imaginary parts of $w$ as
\[
\begin{aligned}
& dx=(1-x^2+y^2) \, dY^1_t-2xy \, dY^2_t-2x \, dY^3_t \\
& dy=-2xy \, dY^1_t+(1+x^2-y^2) \, dY^2_t-2y \, dY^3_t.
\end{aligned}
\]
By Ito's formula, the corresponding second order operator is found to be
\[
\frac12 \frac{\pa ^2S}{\pa x^2}[(1-x^2+y^2)^2+4x^2y^2+4x^2]
+\frac12 \frac{\pa ^2S}{\pa y^2}[(1+x^2-y^2)^2+4x^2y^2+4y^2]
\]
\[
+\frac{\pa ^2S}{\pa x \pa y}[-2xy(1-x^2+y^2)^2-2xy (1+x^2-y^2)+4xy]
= \frac12 (1+x^2+y^2)^2\left(\frac{\pa ^2S}{\pa x^2}+\frac{\pa ^2S}{\pa y^2}\right),
\]
as was claimed.
\end{proof}

\begin{remark} Thus equation \eqref{eqqufiBlinnind2wP} gives a method to express the curvilinear
$2$-dimensional Brownian motion on a sphere in terms of the $3$-dimensional standard (plain) Brownian motion.
\end{remark}

It is natural to ask what is the general class of the triples of operators $L_1,L_2, L_3$, where the same effects hold.

Reducing the attention to the case of self-adjoint matrices $L_j$ let us write them as
\[
L_j=\left(\begin{aligned}
& l^{00}_j \quad l^{01}_j \\
& \bar l^{01}_j \quad l^{11}_j
\end{aligned}
\right), \quad j=1,2,3,
\]
with $l^{00}_j, l^{11}_j \in \R$, $l^{01}_j\in \C$. Let us introduce the $3$-dimensional real vectors
$L^0, L^1, L^{\de}, L_j^R, L^I$ defined by their coordinates
\[
L_j^0=l_j^{00}, \quad L_j^1=l_j^{11}, \quad L_j^{\de}=(l_j^{11}-l_j^{00})/2, \quad
 L_j^R={Re} \, l_j^{01}, \quad L_j^I={Im} \, l_j^{01}.
 \]

\begin{prop}
\label{proponspecardif2dim}

(i) The second order part of the diffusion operator arising from the stochastic equation
\eqref{eqqufiBlinnind2w} is isothermic, that is, it has the form
\[
\om(x,y)   \left(\frac{\pa ^2S}{\pa x^2}+\frac{\pa ^2S}{\pa y^2}\right)
\]
with some positive function $\om(x,y)$ if and only if the vectors
$L^{\de}, L_j^R, L^I$ form an orthonormal basis in $\R^3$, up to a common constant multiplier.
If this is the case, then this operator actually coincides with the Laplace-Beltrami operator
\eqref{eqquaqdrfromPau} (again of course up to a constant multiplier).

 (ii)  The whole diffusion operator arising from the stochastic equation
\eqref{eqqufiBlinnind2w} with vanishing $H$ is isothermic (that is, additionally to (i),
all the first order terms cancel as in the case of the Pauli matrices) if and only if
  the vectors $L^{\de}, L_j^R, L^I$ form an orthonormal basis in $\R^3$  (up to a common constant multiplier)
 and $L^0=-L^1$. Moreover, under these conditions the diffusion operator of
 stochastic equation \eqref{eqqufiBlinnind2w} coincides with the diffusion operator
 arising from   equation \eqref{eqqufiBlinnind2w} rewritten in terms of the innovation process.
\end{prop}

\begin{proof}
This is done by lengthy explicit calculations, which we omit.
\end{proof}

Since the transpose of an orthogonal matrix (in our case the matrix with the columns built
from the vectors  $L^0, L_j^R, L^I$) is also orthogonal, Proposition \ref{proponspecardif2dim}
can be formulated in the following more transparent way.

\begin{prop}
\label{proponspecardif2dimr}
The diffusion operator arising from equation \eqref{eqqufiBlinnind2w} with the 3 self-adjoint matrices $L_j$
coincides with the Laplacian on a sphere (up to a multiplier),
if and only if three matrices $L_j$ form a basis in the space of traceless self-adjoint matrices,
which is orthogonal in the sense that
\[
{\tr} (L_j L_k)=2(l_j^{00} l_k^{00} +{Re} \, l_j^{01} \, {Re} l_k^{01} + {Im} \, l_j^{01} \, {Im} l_k^{01})
 =a \de_{jk}
\]
with a constant $a$. In the case of the Pauli matrices $a=2$. The exact Laplacian arises from $a=1$.
\end{prop}

\section{Special homodyne detections leading to the Laplace-Beltrami operator on projective spaces (for qudits)}
\label{qugametoprsp}

In this section we extend the previous results to quantum systems in $\C^{n+1}$ with arbitrary $n$
(a qudit with $d=n+1$).

As in the case of qubit, let us start by writing the corresponding filtering equation \eqref{eqqufiBlinn}
in terms of the vector $W=(1,w_1,\cdots, w_n)=\chi/\chi_0$, that is, in the projective coordinates $w_1,\cdots, w_n$.
We have

  \begin{equation}
\label{eqqufiBlinn8}
d\chi_k =-[iH\chi +\frac12 \sum_j L_j^*L_j \chi ]_k\,dt+\sum_j (L_j\chi)_k dY^j_t, \quad k=0,\cdots, n.
\end{equation}
Hence by the Ito formula
\[
d\chi_0^{-1} =\frac{1}{\chi_0^2} [i(H\chi)_0+\frac12 \sum_j (L_j^*L_j\chi)_0] \,dt
-\frac{1}{\chi_0^2}\sum_j (L_j \chi)_0 \, dY^j_t
+\frac{1}{\chi_0^3}\sum_j (L_j \chi)_0^2 \, dt.
\]
Consequently, by the Ito product rule, we find for $k>0$ that
\[
dw_k=w_k [i(HW)_0+\frac12 \sum_j (L_j^*L_jW)_0] \,dt
-w_k\sum_j (L_jW)_0 \, dY^j_t
+w_k\sum_j (L_jW)_0^2 \, dt
\]
\[
-[i(HW)_k +\frac12 \sum_j (L_j^*L_jW)_k]\,dt+\sum_j (L_jW)_k dY^j_t
-\sum_j (L_jW)_0  (L_jW)_k \, dt,
\]
and thus the {\it quantum filtering equation for qudits} (with $d=n+1$) in terms of the projective coordinate $W$:
 \[
dw_k =i[w_k(HW)_0-(HW)_k]\, dt+\frac12 [\sum_j w_k(L_j^*L_jW)_0- (L_j^*L_jW)_k ] \, dt
\]
 \begin{equation}
\label{eqqufiBlinnind3w}
+\sum_j[w_k(L_jW)_0^2 -(L_jW)_0 (L_jW)_k]\, dt
+\sum_j[(L_jW)_k -w_k(L_jW)_0]\, dY^j_t.
\end{equation}

To reduce complexity, let us discuss in more detail the case of a three-dimensional Hilbert space (a qutrit)
of the vectors $\chi=(\chi_0,\chi_1, \chi_2)$ (the general case being quite similar).
Extending the case of qubit it is natural to choose $L_j$ to be the $8$ generalized Pauli or Gell-Mann matrices:
3 symmetric $\si^s_{jk}$ (with $1$ on places $jk$ and $kj$ and zero otherwise ), $0\le j<k\le 2$,
3 antisymmetric  $\si^a_{jk}$ (with $-i$ on the place $jk$ and $i$ on the place $kj$, and zero othersise),
 $0\le j<k\le 2$, and 2 diagonal matrices $\si^d_k$:
 \[
 \si^d_1=\left(
 \begin{aligned}
 & 1 \quad \quad 0 \quad 0 \\
 & 0 \quad -1 \quad 0 \\
 & 0 \quad \quad 0 \quad 0
 \end{aligned}
 \right),
 \quad  \si^d_2=\frac{1}{\sqrt 3} \left(
 \begin{aligned}
 & 1 \quad  0 \quad \quad 0 \\
 & 0 \quad 1 \quad \quad 0 \\
 & 0 \quad 0 \quad -2
 \end{aligned}
 \right).
 \]

 Thus
 \[
 \si_{01}^s \chi =\left(
 \begin{aligned}
 & \chi_1 \\
 & \chi_0  \\
 & \,\,  0
 \end{aligned}
 \right),
 \quad
  \si_{02}^s \chi =\left(
 \begin{aligned}
 & \chi_2 \\
 & \, \, 0  \\
 & \chi_0
 \end{aligned}
 \right),
 \quad
 \si_{12}^s \chi =\left(
 \begin{aligned}
 & \, \, 0 \\
 & \chi_2  \\
 & \chi_1
 \end{aligned}
 \right),
 \quad
 \si_{01}^a \chi =\left(
 \begin{aligned}
 & -i \chi_1 \\
 & \quad i \chi_0  \\
 & \,\,  \quad 0
 \end{aligned}
 \right),
 \]
 \[
  \si_{02}^a \chi =\left(
 \begin{aligned}
 & -i \chi_2 \\
 & \, \, \quad  0  \\
 & \quad i \chi_0
 \end{aligned}
 \right),
 \quad
 \si_{12}^a \chi =\left(
 \begin{aligned}
 & \, \, \quad 0 \\
 & -i \chi_2  \\
 & \quad i \chi_1
 \end{aligned}
 \right),
 \quad
 \si_1^d \chi =\left(
 \begin{aligned}
 & \quad \chi_0 \\
 & -\chi_1  \\
 & \,\,  \quad 0
 \end{aligned}
 \right),
 \quad
  \si_2^d \chi =\frac{1}{\sqrt 3}\left(
 \begin{aligned}
 & \quad \chi_0 \\
 & \quad \chi_1  \\
 & -2 \chi_2
 \end{aligned}
 \right).
 \]

 In arbitrary dimension it is more convenient to work directly in complex coordinates $w_k, \bar w_k$
 (rather than playing with their real and imaginary parts).
 Again direct substitution of the above Gell-Mann matrices into the equation \eqref{eqqufiBlinnind3w}
 (we omit the lengthy by direct calculations) shows the following analog of Proposition \ref{proponPaulidif}.

 \begin{prop}
 \label{GellManncancel}
  Equation \eqref{eqqufiBlinnind3w} with $L_j$ chosen as the 8 Gell-Mann matrices and written
  for vanishing $H$ takes the form
  \begin{equation}
\label{eqqufiGellM}
\begin{aligned}
 & dw_1=(1-w_1^2) \, dY_t^{01,s} - w_1 w_2 \, dY_t^{02,s}+w_2 \, dY_t^{12,s} \\
& \quad\quad  +i(1+w_1^2) \, dY_t^{01,s} +i w_1 w_2 \, dY_t^{02,a}-i w_2 \, dY_t^{12,a}
 -2w_1 \, dY_t^{1,d} \\
& dw_2=-w_1w_2 \, dY_t^{01,s} +(1- w_2^2) \, dY_t^{02,s}+w_1 \, dY_t^{12,s} \\
& \quad\quad  +iw_1w_2 \, dY_t^{01,s} +i (1+ w_2^2) \, dY_t^{02,a}+i w_1 \, dY_t^{12,a}
 -w_2 \, dY_t^{1,d}-\sqrt 3 w_2  \, dY_t^{2,d}.
 \end{aligned}
 \end{equation}
 (with all terms with $dt$ vanishing), and exactly the same form has this equation when rewritten in terms
 of the innovation process, which is now an $8$-dimensional standard Wiener process with the coordinates
 \[
 dB_t^{jk,s}= dY_t^{jk,s}-2\langle \si_{jk}^s \rangle_W \, dt,
 \quad
 dB_t^{jk,a}= dY_t^{jk,a}-2\langle \si_{jk}^a \rangle_W \, dt,
 \quad
 dB_t^{k,d}= dY_t^{k,d}-2\langle \si_{k}^d \rangle_W \, dt.
 \]
 Finally the diffusion operator $D$ arising from the stochastic differential
 equation \eqref{eqqufiGellM} has the form  $D=2 \De_{pro}$, where $ \De_{pro}$ is the
 major (second order) part of the Laplace-Beltrami operator on the complex projective space $P\C^2$:
   \[
 \De_{pro}S(w_1,w_2)=(1+\sum_j|w_j|^2)
 \bigl[( 1+|w_1|^2)\frac{\pa^2S}{\pa w_1 \pa \bar w_1}
 +( 1+|w_2|^2))\frac{\pa^2S}{\pa w_2 \pa \bar w_2}
 \]
  \begin{equation}
\label{eqLaplRiemProjcom}
+w_1\bar w_2 \frac{\pa^2S}{\pa w_1 \pa \bar w_2}+\bar w_1 w_2 \frac{\pa^2S}{\pa \bar w_1 \pa w_2}
\bigr].
 \end{equation}
  \end{prop}

Of course there exists a characterization of all collections of $L_j$ with the same property,
analogous to Proposition \ref{proponspecardif2dimr}.

For a quantum system in $\C^{n+1}$ there are $(n^2+2n)$ generalized Pauli matrices.
Choosing these matrices as the coupling operators in a homodyne detection scheme will lead analogously
to the invariant BM on the complex projective space $P\C^n$.

\section{Theory of drift control and games on Riemannian manifolds}
\label{secgameon Riemann}

Now we shall develop the theory of the classical or mild solutions to the Hamilton-Jacobi-Bellman-Isaacs
equations arising in the stochastic control and differential games on compact Riemannian manifolds
with a controlled drift and the fixed underlying Markov process being the standard Brownian motion on $M$.
In the next section it will be used to build the theory of dynamic quantum games that can be reduced to
such stochastic games under the special arrangement homodyne detection as shown in previous sections.
Let
  \begin{equation}
\label{Lapgeom}
\De_{LB} \phi ={div} \, (\nabla \phi) =\frac{1}{\sqrt {\det g}}
\sum_{j,k}\frac{\pa}{\pa x_j} \left(\sqrt {\det g} \, g^{jk}\frac{\pa}{\pa x_k}\right)
 \end{equation}
 denote the Laplace-Beltrami operator on a compact Riemannian manifold $(M,g)$ of dimension $N$,
with the Riemannian metric given by the matrix $g=(g_{jk})$ and its inverse matrix $G=(g^{jk})$.
Let $K(t,x,y)$ be the corresponding heat kernel, that is, $K(t,x,y)$ is the solution of the corresponding heat equation
$(\pa K/\pa t)=\De_{LB} K$ as a function of $(t>0,x\in M)$ and has the Dirac initial condition $K(0,x,y)=\de_y(x)$.
It is well known that the Cauchy problem for this heat equation is well posed in $M$ and the resolving operators
 \begin{equation}
\label{eqheatsem}
S_tf(x)=\int_M K(t,x,y) f(y) \, dv(y),
\end{equation}
where $dv(y)$ is the Remannian volume on $M$, form a strongly continuous semigroup of contractions
(the Markovian semigroup of the Brownian motion in $M$) in the space $C(M)$
of bounded continuous functions on $M$, equipped with the sup-norm. Let $C^1(M)$ denote the space of
continuously differentiable functions on $M$ equipped with the norm
\[
\|f\|_{C^1(M)}=\sup_x|f(x)|+\sup_x \|\nabla f(x)\|_M,
\]
where in local coordinates
\[
\|\nabla f(x)\|_M^2=\left(\nabla f(x), G(x)\nabla f(x)\right)
=\sum_{jk} g^{jk} \frac{\pa f}{\pa x_j} \frac{\pa f}{\pa x_k}.
\]

The key properties of this semigroup needed for our theory are the following smoothing
and smoothness preservation properties.

 \begin{prop}
 \label{propheatsemomansmoo}
 (i) The operators $S_t$ are smoothing:
   \begin{equation}
\label{eq1propheatsemomansmoo}
\|S_tf\|_{C^1(M)}\le C t^{-1/2} \|f\|_{C(M)}
 \end{equation}
 with a constant $C$, uniformly for any compact interval of time.

   (ii) The operators $S_t$ are smoothness preserving:
   \begin{equation}
\label{eq2propheatsemomansmoo}
\|S_tf\|_{C^1(M)}\le C \|f\|_{C^1(M)}
 \end{equation}
 with a constant $C$, uniformly for any compact interval of time.
\end{prop}

\begin{remark} This result is possibly known, but the author did not find any precise reference.
It is standard for diffusions in $\R^d$, but seemingly not so standard for manifolds.
We sketch a proof briefly. An alternative proof of (ii) (by-passing estimates from (i))
and its extension to higher derivatives can
be built on the theory of SDEs on $(M,g)$.
\end{remark}

\begin{proof} (i) This is a consequence of the well known estimate for the derivatives of the heat kernel
on a compact Riemannian manifold (see Theorem 6 in \cite{Davies89}):
 \begin{equation}
\label{eqheatsemder}
\|\nabla K(t,x,y)\|_M \le C(\de,N) t^{-N/2} t^{-1/2} \exp \left\{ -\frac{d^2(x,y)}{(4+\de)t}\right\},
\end{equation}
with any $\de>0$ and a constant $C(\de,N)$, where the derivative $\nabla$ is taken with respect to $x$,
and where $d$ is the Riemannian distance in $M$.
In fact, differentiating \eqref{eqheatsem} and using \eqref{eqheatsemder} yields \eqref{eq1propheatsemomansmoo}.

(ii) This follows from \eqref{eqheatsemder} and the method of parametrix (frozen coefficients) approximation.
This method (see e.g. formula (5.60) in \cite{Kolbook19}) starts by representing $K$ in terms of its asymptotics $K_{as}$
and the integral correction as
   \begin{equation}
\label{eq3propheatsemomansmoo}
K(t,x,y)= K_{as}(t,x,y)+\int_0^t K(t-s,x,z)F(s,z,y) \, ds,
\end{equation}
where $F$ is the error term in the equation for $K_{as}$, that is
\[
\frac{\pa K_{as}}{\pa t}(t,x,y)-\De_{LB} K_{as}(t,x,y)=-F(t,x,y).
\]
From \eqref{eqheatsemder} it follows that the derivative of the second term in \eqref{eq3propheatsemomansmoo}
is bounded and thus the estimate for the derivative reduces to the derivatives arising from $K_{as}$,
and these estimates are standard and are performed as in the case of heat equations in $\R^d$.
\end{proof}

For the stochastic control of diffusions on $(M,g)$ with the second order part being fixed as $\De_{LB}$,
and where control is carried out via the drift only,
the corresponding HJB equation is the equation
\begin{equation}
\label{eqHJBmanif}
\frac{\pa f}{\pa t}= \De_{LB} f +H(x,\nabla f (x)),
\end{equation}
where the Hamiltonian function is of the form
\begin{equation}
\label{eqHJBmanif1}
H(x,p)=\sup_{u\in U} [g(x,u)p+J(x,u)],
\end{equation}
where $U$ is the set of possible controls and $g,J$ are some continuous functions.
In case of zero-sum stochastic two-player games with the so-called Isaac's condition
the Hamiltonian function takes the form
\begin{equation}
\label{eqHJBmanif2}
H(x,p)=\sup_{u\in U}\inf_{v\in V} [g(x,u,v)p+J(x,u,v)]=\inf_{v\in V}\sup_{u\in U} [g(x,u,v)p+J(x,u,v)].
\end{equation}
The possibility to exchange $\sup$ and $\inf$ here is called Isaac's condition. It is fulfilled, in particular,
when the control of two players can be separated in the sense that the Hamiltonian becomes
\begin{equation}
\label{eqHJBmanif3}
H(x,p)=\sup_{u\in U} [g_1(x,u)p+J_1(x,u)]+\inf_{v\in V}[g_2(x,v)p+J_2(x,v)]+J_0(x).
\end{equation}

It is worth recalling here that though the theory of HJB is often built (for simplicity)
 for the Cauchy problem of equation \eqref{eqHJBmanif} in forward time, in the control theory
 it appears more naturally as the backward Cauchy problem for the equation
 \begin{equation}
\label{eqHJBmanifbac}
\frac{\pa f}{\pa t}+ \De_{LB} f +H(x,\nabla f (x))=0, \quad t\in [0,T],
\end{equation}
with a given terminal condition $f_T$ at some time $T$. This way of writing the HJB equation
becomes unavoidable whenever any of the parameters of the problem are explicitly time dependent.

Let us now consider the general Hamilton-Jacobi-Bellman -Isaacs equation \eqref{eqHJBmanif}
 with $H$ being a Lipschitz continuous
function of its two variables. It is well known (and easy to see) that if $f$ is a classical solution of
\eqref{eqHJBmanif} with the initial condition $Y$, then $f$ solves also the following integral equation
\begin{equation}
\label{eqdefHJBsmoothmild}
f_t=e^{t\De_{LB}}Y +\int_0^t e^{(t-s)\De_{LB}} H\left(.,\frac{\pa f_s}{\pa x}(.)\right) \, ds,
\end{equation}
referred to as the {\it mild form} of \eqref{eqHJBmanif}. Solutions to the mild equation
\eqref{eqdefHJBsmoothmild} (which may not solve \eqref{eqHJBmanif}, because of the lack of sufficient smoothness)
are often referred to as {\it mild solutions} to \eqref{eqHJBmanif}.

The following result gives the well-posedness of the HJB-Isaacs equation  \eqref{eqHJBmanif}
with explicit estimates for the growth of solutions and their continuous dependence on initial data.
\begin{theorem}
\label{thHJBsmoothwel}
 Let $H(x,p)$ be a continuous function on the cotangent bundle $T^*M$ to the compact Riemannian manifold $(M,g)$ such that
\begin{equation}
\label{eq1athHJBsmoothwel}
|H(x,p_1)-H(x,p_2)| \le L_H \|p_1-p_2\|_M
\end{equation}
with a constant $L_H$. Then for any $Y\in C^1(M)$ there exists
 a unique solution $f_. \in C([0,T], C^1(M))$
of equation \eqref{eqdefHJBsmoothmild}. Moreover, for all $t\le T$,
 \[
\|f_t(Y)-Y\|_{C^1(M)} \le E_{1/2}(C L_H \Ga(1/2)t^{1/2})
\]
 \begin{equation}
\label{eq3athHJBsmoothwel}
\times \left(2t^{1/2} C (h +L_H\|Y\|_{C^1(M)})+ \|(e^{t\De_{LB}}-1)Y\|_{C^1(M)}\right),
\end{equation}
where $h=\sup_x|H(x,0)|$ and $C$ is from Proposition  \ref{propheatsemomansmoo}.
and the solutions $f_t(Y_1)$ and $f_t(Y_2)$ with different initial data $Y_1,Y_2$ enjoy
the estimate
 \begin{equation}
\label{eq4athHJBsmoothwel}
\|f_t(Y_1)-f_t(Y_2)\|_{C^1(M)} \le C \|Y_1-Y_2\|_{C^1(M)} E_{1/2}(C L_H \Ga(1/2)t^{1/2}),
\end{equation}
where $E_{1/2}$ denotes the Mittag-Leffler function.
\end{theorem}

The proof of the theorem is identical to the corresponding proof given for the equations in $\R^d$ in
\cite{Kolbook19} (Section 6.1) and it follows essentially from the fixed point argument and
Proposition \ref{propheatsemomansmoo}.

Assuming additionally that $H$ is Lipshitz continuous in the first argument so that
\begin{equation}
\label{eq2thHJBsmoothwel1}
|H(x_1,p)-H(x_2,p)| \le L_H d(x_1,x_2) \, \|p\|_M,
\end{equation}
one can improve the result of Theorem \ref{thHJBsmoothwel} by showing
(in exactly the same way as for $\R^d$, see again \cite{Kolbook19}) that with the initial condition $Y$ from $C^2(M)$
the solution $f_t$ will belong to $C^2(M)$ for all $t$ and hence will be a (unique) classical solution to the Cauchy problem
of the HJB-Isaacs equation  \eqref{eqHJBmanif}.

Finally, the standard result of the stochastic control theory
(called the verification theorem, see e.g. \cite{FlemSon}) states that
a classical solution to the HJB-Isaacs equation yields in fact
the optimal cost for the corresponding stochastic control problem
or the minimax solution in case of zero-sum games.

\section{Towards the theory of dynamic games under the special homodyne detections}
\label{secqugameth}

Applying the results of the previous section in conjunction with the detection schemes of Sections
\ref{qugametosphere}, \ref{qugametoprsp} leads automatically to the theory of dynamic
quantum games under these detetion schemes.

In fact, assume that the hamiltonian operator $H$ decomposes into the three parts $H=H_0+uH_1+vH_2$,
where $H_0$ is the Hamiltonian operator of the "free" motion of an atom and the strength $u$ and $v$
of the application of the operators
$H_1$ and $H_2$ can be chosen strategically by the two players I and II respectively.
To simplify the formulas let us assume that allowed values of the control parameters lie
in certain symmetric intervals: $u\in [-U,U]$, $v\in [-V,V]$ with some constants $U,V\ge 0$,
The case of a pure control (not a game) corresponds to the choice $V=0$ and is thus automatically included.
Assume that players $I$ and $II$ play a standard dynamic zero-sum game with a finite time horizon $T$
meaning that the objective of $I$ is to maximize the payoff
 \begin{equation}
\label{eqcostfun}
P(t,W; u(.), v(.)) =\E \int_t^T    \langle J\rangle_{W(s)} \, ds +\langle F\rangle_{W(T)},
\end{equation}
where $J$ and $F$ are some operators expressing the current and the terminal costs of the game
(they may depend on $u$ and $v$, but we exclude this case just for simplicity) and $\E$ denotes the expectation
with respect to
the random trajectories $W(s)$ arising from dynamic \eqref{eqqufiBlinnind3w} under the strategic choice of
the controls $u$ and $v$ by the players.

The Isaacs condition \eqref{eqHJBmanif2} is fulfilled under our assumptions. Assuming
$\{L_j\}$ are chosen with our special detection scheme such that the corresponding diffusion operator
with vanishing $H$ coincides with the second order part of the Laplace-Beltrami operator on the complex projective space,
the HJB-Isaacs equation \eqref{eqHJBmanifbac} for the minimax value
 \begin{equation}
\label{eqcostfunopt}
S(t,W) =\max_{u(.)}\min_{v(.)}P(t,W; u(.), v(.))=\min_{v(.)} \max_{u(.)} P(t,W; u(.), v(.))
\end{equation}
of the game  takes the form
\[
0=\frac{\pa S}{\pa t} +(\al(W), \nabla S)+\De_{LB}S+\langle J\rangle_W
\]
\[
+\sup_u \left\{u \sum_k \left[ {Re}[iw_k(H_1W)_0-i(H_1W)_k]\frac{\pa S}{\pa x_k}
+ {Im} [iw_k(H_1W)_0-i(H_1W)_k]\frac{\pa S}{\pa y_k}\right]\right\}
\]
 \begin{equation}
\label{eqHJBIsforspecialhom}
+\inf_v\left\{ v\sum_k \left[{Re} [iw_k(H_2W)_0-i(H_2W)_k]\frac{\pa S}{\pa x_k}
+ {Im} [iw_k(H_2W)_0-i(H_2W)_k]\frac{\pa S}{\pa y_k}\right]\right\},
\end{equation}
where $\al$ includes the contributions arising from $H_0$ and from the 1st order terms
of the Laplace-Beltrami operator (if any). Equivalently, it can be rewritten as
\[
0=\frac{\pa S}{\pa t} +(\al(W), \nabla S)+\De_{LB}S+\langle J\rangle_W
\]
\[
+U \left| \sum_k \left[{Re} [iw_k(H_1W)_0-i(H_1W)_k]\frac{\pa S}{\pa x_k}
+ {Im} [iw_k(H_1W)_0-i(H_1W)_k]\frac{\pa S}{\pa y_k}\right] \right|
\]
 \begin{equation}
\label{eqHJBIsforspecialhom1}
-V \left| \sum_k \left[ {Re} [iw_k(H_2W)_0-i(H_2W)_k]\frac{\pa S}{\pa x_k}
+ {Im} [iw_k(H_2W)_0-i(H_2W)_k]\frac{\pa S}{\pa y_k}\right]\right|.
\end{equation}

According to the theory from the previous section, under, the backward Cauchy problem
for this equation with the terminal condition $\langle F\rangle_{W}$ has the unique classical solution that yields
the minimax value of our zero-sum game.

\section{Special homodyne detections leading to standard BM on tori}
\label{secqugametotori}

An alternative approach to simplify \eqref{eqqufiBlinnind3w} is to choose $L_j$ in a way allowing
for invariant manifolds and thus reducing the complexity (the dimension) of resulting controlled process.
Suppose we choose the coordinates $\chi$, where $H$ is diagonal, that is $h_{km}=0$ for $k\neq m$.
 It turns out (which seen by direct inspection) that if one chooses
$n$ operators $L_j$, $j=1,\cdots, n$, as anti-Hermitian
diagonal operators with only one non-vanishing elements, namely, with the entries
\[
(L_j)^{km}=ir_j \de^j_k\de^j_m
\]
with some real numbers $r_j$, then \eqref{eqqufiBlinnind3w} decomposes into the system of uncoupled equations
 \begin{equation}
\label{eqqufiBldecoup}
dw_k =iw_k(h_{00}-h_{kk}) \, dt -\frac12 r_k^2
w_k \, dt +ir_k w_k \, dY^k_t, \quad k=1, \cdots, n,
\end{equation}
from which it follows that $d(w_k\bar w_k)=0$ for all $k$, and thus the whole process lives on
the $n$-dimensional torus
$T_n^C=\{(w_1, \cdots, w_n):|w_k|=C_k\}$,
where the vector $C=\{c_m\}$ is specified by the initial condition.

Moreover, since all $L_j$ are anti-Hermitian, the innovation process $B_t=(B^k_t)$ coincides with
the output process $Y_t$, both forming the standard Brownian motion.

\begin{remark} Therefore, in case of a qubit, only one operator $L$ (with an element $ir$ in the right low corner
and zeros otherwise)
is sufficient to get $d(w\bar w)=0$. It is not difficult to see
that for a qubit any operator leading to this effect has the form $L+\al \1$ with a constant $\al \in \C$.
\end{remark}

Writing $w_k=|w_k|\exp\{i\phi_k\}$ we can rewrite \eqref{eqqufiBldecoup} (using Ito's formula, of course)
in terms of the angles as follows
  \begin{equation}
\label{eqqufiBldecoupan}
d\phi_k =(h_{00}-h_{kk}) \, dt +r_k \, dY^k_t, \quad k=1, \cdots, n,
\end{equation}
which is the BM on the torus $T_n^C$ with a constant drift.

This is of course simpler, than the invariant BM on the projective space $P\C^n$ obtained by choosing
$(n^2+2n)$ generalized Pauli operators. However, from the point of the application to control and games,
this homodyne detection can be used only in case when under any choice of control parameters $u\in U$
the resulting family of possible Hamiltonians $H(u)$ is a commuting family. Only in this case we can choose a basis where
all $H(u)$ are diagonal and thus treat them all with a single choice of anti-Hermitian operators $L_j$ above. The general theory
of games is then exactly the same as in Section \ref{secqugameth}, with the projective spaces substituted by the tori.
An exactly solvable example of such case will be given below.

\section{Example of exactly solvable model}
\label{secqugametotoriex}

In the case of a qubit, equation \eqref{eqqufiBldecoupan} reduces to the equation
\begin{equation}
\label{eqqufiBldecoupqub}
d\phi =(h_0-h_1) \, dt +r \, dB_t,
\end{equation}
with real constants $h_1, h_2,r$ and the standard one-dimensional BM $B_t$, thus describing the
BM on a circle with a drift.

Let us now choose $r=1$ and assume that two players I and II can control the strength and direction of the field
yielding the first and the second entries $h_0$ and $h_1$ respectively. Assume further that  the goal of player I
is to maximize the integral cost
\[
\int_t^T    \langle J\rangle_{W(s)} \, ds  +\langle F\rangle_{W(T)}
\]
with $W(t)=C e^{i\phi(t)}$ and some Hermitian operators $J=(J_{jk})$ and $F=(F_{jk})$, so that
\[
  \langle J\rangle_{W}=\frac{J_{00}+J_{01} w+\bar J_{01}\bar w +J_{11}|w|^2}{1+|w|^2}
  =\frac{J_{00}+2C \, {Re} (J_{01}C e^{i\phi}) +J_{11}C^2}{1+C^2}.
\]
Ignoring irrelevant constants the current cost function $\langle J\rangle_{W}$ rewrites as
$(a \cos\phi+b\sin \phi)$, or, by shifting $\phi$, even simpler as just $\cos \phi$.
Hence the corresponding HJB-Isaacs equation for this game becomes
 \[
\frac{\pa S}{\pa t} +\frac12\frac{\pa ^2S}{\pa \phi^2}+\cos \phi
+\max_u [uh_0\frac{\pa S}{\pa \phi}]-\max_v [vh_1\frac{\pa S}{\pa \phi}]=0.
\]
Assuming as above that $u$ and $v$ can be chosen from some intervals, this equation rewrites as
\begin{equation}
\label{eqHJBIsexact}
\frac{\pa S}{\pa t} +\frac12\frac{\pa ^2S}{\pa \phi^2}+\cos \phi
+\al \left|\frac{\pa S}{\pa \phi}\right|=0,
\end{equation}
with a real constant $\al$.

\begin{remark} Player I has an advantage if $\al>0$.
\end{remark}

Instead of the fixed time horizon problem, let us consider the corresponding stationary problem,
when one looks for the average payoff per unit time for a long lasting game. Analytically this means
searching for a solution to equation \eqref{eqHJBIsexact} with $S$ linearly dependent on $t$:
$S(t,\phi)=\la (T-t)+S(\phi)$ with some function $S(\phi)$. Here $\la$ is the average payoff per unit time and
$S(\phi)$ solves the stationary HJB equation
\begin{equation}
\label{eqHJBIsexactst}
\frac12 S''+\al |S'| +\cos \phi -\la=0.
\end{equation}

From the symmetry of the problem it is clear that $S(\phi)$ is an even periodic function of $\phi$,
and thus it has the boundary conditions
$S'(0)=S'(\pi)=0$. Moreover, $S$ is decreasing in $\phi$ on the interval $\phi\in [0,\pi]$
(since $\cos \phi$ has maximum at $\phi=0$),
so that on this interval equation \eqref{eqHJBIsexactst} turns to the equation
\begin{equation}
\label{eqHJBIsexactst1}
\frac12 S''-\al S' +\cos \phi -\la=0.
\end{equation}
Hence for the function $f=S'$ we get the problem
\begin{equation}
\label{eqHJBIsexactst2}
f'-\al f +\cos \phi -\la=0, \quad f(0)=f(\pi)=0.
\end{equation}
 The stationary solution $S$ itself is defined up to a constant, as it should be,
 since what one is looking for is actually the average cost $\la$.

 The solution to equation in \eqref{eqHJBIsexactst2} with the initial condition $f(0)=0$ is
\begin{equation}
\label{eqHJBIsexactst3}
f(\phi)=\int_0^{\al \phi} e^{\al(\phi-s)}(\la -\cos s) \, ds.
\end{equation}
Hence, from $f(\pi)=0$ one finds
\begin{equation}
\label{eqHJBIsexactst4}
\la =\frac{\int_0^{\pi} e^{\al(\pi-s)}\cos s  \, ds}{\int_0^{\pi} e^{\al(\pi-s)} \, ds}
=\frac{\al^2}{\al^2+1} \,\, \frac{e^{\al\pi}+1}{e^{\al\pi}-1}.
\end{equation}

Similarly one can solve the infinite horizon problem with discount, for which one searches for
$S$ of the form $S(t,\phi) =e^{-t\de}S(\phi)$, with a fixed discount factor $\de>0$, and the stationary HJB equation
writes down as
\begin{equation}
\label{eqHJBIsexactdi}
\frac12 S''+\al |S'| +\cos \phi -\de S=0.
\end{equation}
The same monotonicity and evenness conditions as above lead now to the problem
 \begin{equation}
\label{eqHJBIsexactdi1}
\frac12 S''-\al S' +\cos \phi -\de S=0, \quad S'(0)=S'(\pi)=0
\end{equation}
on the interval $[0,\pi]$.
This linear problem is easy to solve. First one finds the general solution of the equation in the form
\[
S=Ae^{a_1\phi}+Be^{a_2\phi}+a\cos \phi +b\sin \phi
\]
with
\[
a_{1,2}=\al \pm \sqrt{\al^2+2\de}, \quad b=\frac{4\al}{4\al^2+(1+2\de)^2}, \quad a=b\frac{1+2\de}{2\al},
\]
and then the boundary condition gives the linear system
\[
A a_1+B a_2+b=0, \quad A a_1 e^{a_1\pi}+B a_2 e^{a_2\pi}-b=0,
\]
from which $A$ and $B$ are found.

\section{Special homodyne detections leading to the standard BM on Euclidean spaces}
\label{secqugametoEuc}

As the third model of special homodyne detections we introduce the arrangements,
under which the resulting filtering process is the standard BM in $\R^n$, that is
a process govern by the plain Laplacian, though perturbed by a drift with unbounded
coefficients. We need here $2n$ operators $L_j$ for a quantum system in $\C^{n+1}$
(unlike  $(n^2+2n)$ Pauli matrices leading to the invariant BM and $n$ matrices
leading to the tori). We shall take $n$ operators $L_{j1}=(l_{j1}^{kl})$ and $n$
operators $L_{j2}=(l_{j2}^{kl})$ $j=1, \cdots, n$, such that the only non-vanishing
entries of their matrices are the elements of the first column $l_{j1}^{k0}$ and
$l_{j2}^{k0}$ with $k=1, \cdots, n$. More concretely, let
\begin{equation}
\label{eqchoiceEucl}
l_{j1}^{k0}=\de^j_k, \quad l_{j2}^{k0}=i\de^j_k.
\end{equation}

Under this choice $(L_{j1}\chi)_0=(L_{j2}\chi)_0=0$ and
\[
L_{j1}^*L_{j1}=L_{j1}^*L_{j1}
=\left(
 \begin{aligned}
 & 1 \quad 0 \quad \cdots \quad 0 \\
 & 0 \quad 0 \quad \cdots \quad 0 \\
 &  \quad \quad \cdots  \\
 & 0 \quad 0 \quad \cdots \quad 0
 \end{aligned}
 \right)
\]
for all $j$, and equation  \eqref{eqqufiBlinnind3w} takes the form
 \begin{equation}
\label{eqchoiceEucleq}
dw_k =i[w_k(h_{00}+\sum_l h_{0l}w_l)-h_{k0}-\sum_l h_{kl} w_l]\, dt+nw_k \, dt
+dY^{k1}_t+i\, dY^{k1}_t.
\end{equation}

The corresponding diffusion operator for vanishing $H$ gets the form
 \begin{equation}
\label{eqchoiceEucloper}
DS(x,y)=n\sum_k \left(x_k\frac{\pa S}{\pa x_k}+y_k\frac{\pa S}{\pa y_k}\right)
+\sum_k \left(\frac{\pa ^2S}{\pa x_k^2}+\frac{\pa ^2S}{\pa y_k^2}\right)
\end{equation}
in the real coordinates $x_k,y_k$
such that $w_k=x_k+iy_k$, that is, it defines a Gaussian (Ornstein-Uhlenbeck)
diffusion in $\R^{2n}$ and its major second order part is just the standard Laplacian in $\R^{2n}$.

From the first sight this third homodyne arrangement seems to be the simplest one. However,
the catch is that, unlike the cases of compact spaces above (projective spaces and tori), where
all coefficients are automatically bounded, here the Hamiltonian controlled part has unbounded
drift (generally speaking of quadratic growth), which complicates the investigation of the corresponding
HJB equations in $\R^{2n}$. In the case of a commuting set of controlled Hamiltonians, when all matrices
$H$ has only diagonal coefficients, the quadratic term disappears and the controlled drift coefficient
grows linearly, which makes it amenable to analysis, see e.g. \cite{DiffGameLinearDrift15}. In this paper
we avoid dealing systematically
with unbounded coefficients and will not develop a theory of control in this case.

\begin{remark}
%(i) Though the dynamics arising from the above scheme can be linear (if all Hamiltonian operators commute),
%the resulting control problem will be not linear quadratic under the natural cost functions, because such
% functions are given by the mean values of some operators on the states, which are quadratic with respect to
%normalized states, and not with respect to the projective coordinates $w_k$.
%(ii)
Unlike the previous cases with the projective spaces and tori, in the scheme leading to
\eqref{eqchoiceEucloper} the diffusion operator written in terms of the innovation processes will
be different from  \eqref{eqchoiceEucloper} leading to another complication of the theory, and even to
possibly two different formulations of the control problems.
\end{remark}

\section{Zero-sum games on two coupled atoms}
\label{sectwoatoms}

The examples above can be looked at the dynamic extensions of the initial game of Meyer \cite{MeyerD99}
in the sense that two players are acting on the same atom. In the popular EWL protocol \cite{EWL99}
and the MW protocol \cite{MW00} the players act simultaneously on two different qubits, with the interaction
between the qubits taken into account by choosing entangled initial states and taking appropriate measurement.

The theory of Section \ref{secqugameth} is general enough to accommodate games of two players on different atoms,
moreover with the genuine  interaction of atoms taken into account.
Namely, suppose the atoms of players I and II, playing a zero-sum game, are two quantum systems, in $\C^{n+1}$ each.
The combined Hilbert space is thus $\C^{n+1}\otimes \C^{n+1}$, so that its vectors can be written
as $\chi=\sum \chi_{jk} e_j\otimes e_k$ with $\{e_k\}$ the standard basis in $\C^{n+1}$.
 Suppose player I (respectively II) can act on the first (resp. second) atom by
the controlled Hamiltonian operators $H_I(u)=(h^I_{jk}(u))$ (resp. $H_{II}(v)=(h^{II}_{jk}(v))$),
and the interaction between the atoms is given by an operator $A=(A_{jk,pq})$.

 \begin{remark}
 The standard physics choice of the interaction is the operator arising from possible exchange of photons,
 $A=a_1^*a_2+a_2^*a_1$, with the annihilation operators $a_1$ and $a_2$ of the two atoms.
 \end{remark}

The filtering equation \eqref{eqqufiBlinn} takes the form

 \begin{equation}
\label{eqqufiBlinntwoatom}
d\chi_{jk} =-i\sum_p (h^I_{jp}(u) \chi_{pk} +h^{II}_{pk}(v)\chi_{jp}) \, dt
-i\sum_{p,q} A_{jk,pq} \chi_{pq} \, dt +\cdots ,
\end{equation}
where by $\cdots$ we denoted the terms arising from the coupling with optical devices or from the uncontrolled
Hamiltonian operators of the atoms. As previously, we rewrite this equation in terms of the projective coordinates
$w_{jk}=\chi_{jk}/\chi_{00}$ as follows (where it is understood that $w_{00}=1$):
 \[
dw_{jk} =i\sum_p [w_{jk}(h^I_{0p}(u) w_{p0} +h^{II}_{p0}(v)w_{0p})
-(h^I_{jp}(u) w_{pk} +h^{II}_{pk}(v)w_{jp})] \, dt
\]
\begin{equation}
\label{eqqufiBlinntwoatompro}
+i\sum_{p,q} (w_{jk}A_{00,pq}-A_{jk,pq}) w_{pq} \, dt +\cdots , \quad j+k>0.
\end{equation}

Choosing for observation our special homodyne detection scheme from Section \ref{qugametoprsp}
(with $(2n+1)^2+2(2n+1)$ generalized Pauli operators) we
get the HJB-Isaacs equation \eqref{eqHJBIsforspecialhom} in the form

\[
0=\frac{\pa S}{\pa t} +(\al(W), \nabla S)+\De_{LB}S+\langle J\rangle_W
\]
\[
+\sum_{j,k,p,q}{Re} [i (w_{jk}A_{00,pq}-A_{jk,pq}) w_{pq}]\frac{\pa S}{\pa x_{jk}}
+\sum_{j,k,p,q}{Im} [i (w_{jk}A_{00,pq}-A_{jk,pq}) w_{pq}]\frac{\pa S}{\pa y_{jk}}
\]
\[
+\sup_u \left\{\sum_{j,k,p} \left[ {Re}[i w_{jk}h^I_{0p}(u) w_{p0}-ih^I_{jp}(u) w_{pk}]\frac{\pa S}{\pa x_{jk}}
+ {Im} [i w_{jk}h^I_{0p}(u) w_{p0}-ih^I_{jp}(u) w_{pk}] \frac{\pa S}{\pa y_{jk}}\right]\right\}
\]
 \begin{equation}
\label{eqHJBIsforspecialhomtwoat}
+\inf_v\left\{\sum_{j,k,p} \left[{Re} [iw_{jk} h^{II}_{p0}(v)w_{0p}-ih^{II}_{pk}(v)w_{jp}]\frac{\pa S}{\pa x_{jk}}
+ {Im} [iw_{jk} h^{II}_{p0}(v)w_{0p}-ih^{II}_{pk}(v)w_{jp}] \frac{\pa S}{\pa y_{jk}}\right]\right\},
\end{equation}
where $\al$ includes the contributions arising from the uncontrolled
Hamiltonian operators (if any)  and from the 1st order terms
of the Laplace-Beltrami operator $\De_{LB} $ on the projective space $P\C^{2n+1}$ (if any; there are no such terms for $n=1$).
We are fully in the setting of Section \ref{secqugameth} implying the well-posedness of the
backward Cauchy problem for the HJB-Isaacs equation \eqref{eqHJBIsforspecialhomtwoat} in classical and
 mild solutions that yield the minimax value of the corresponding zero-sum game.

\section{Non-zero-sum games}
\label{sectwoatomsnonzerosu}

In the previous section zero-sum games of two players were analyzed.
However, the initial  EWL and MW protocols are dealing with more general, non-zero-sum games.
These games can be also accommodated in our setting with continuous observations.
Let us consider for simplicity the case of two players playing on two coupled atoms.
$N$ players on $N$ atoms can be looked at analogously.

As in the previous section,  assume that players I and II can act on two atoms
with the combined Hilbert space  $\C^{n+1}\otimes \C^{n+1}$. To simplify the story, we shall assume
that the Hamiltonian operators $H_I(u)$ and $H_{II}(v)$ depend linearly on their control parameters
$u\in [-U,U]$ and $v\in [-V,V]$.
 Suppose player I (respectively II) can act on the 1st (resp. second) atom by
the controlled Hamiltonian operators $H_I(u)=(h^I_{jk}(u))$ (resp. $H_{II}(v)=(h^{II}_{jk}(v))$),
and the interaction between the atoms is given by an operator $A=(A_{jk,pq})$.
Hence the controlled filtering equation \eqref{eqqufiBlinntwoatompro} will be written as

 \[
dw_{jk} =i\sum_p [w_{jk}(uh^I_{0p} w_{p0} +vh^{II}_{p0}w_{0p})
-(u h^I_{jp} w_{pk} +vh^{II}_{pk}w_{jp})] \, dt
\]
\begin{equation}
\label{eqqufiBlinntwoatompronz}
+i\sum_{p,q} (w_{jk}A_{00,pq}-A_{jk,pq}) w_{pq} \, dt +\cdots , \quad j+k>0.
\end{equation}

Unlike the previous section with a single cost function we assume now that the players have different
cost functions, namely that players $I$ and $II$ aims at maximizing the costs
 \begin{equation}
\label{eqcostfunI}
P^I(t,W; u(.), v(.)) =\E \int_t^T    \langle J^I\rangle_{W(s)} \, ds +\langle F^I\rangle_{W(T)},
\end{equation}
 \begin{equation}
\label{eqcostfunII}
P^{II}(t,W; u(.), v(.)) =\E \int_t^T    \langle J^{II}\rangle_{W(s)} \, ds +\langle F^{II}\rangle_{W(T)},
\end{equation}
 respectively. We again assume for simplicity that the current costs $J^{I,II}$ do not depend on control,
  though this is really not essential.

 If player $I$ acts according to some strategy $u=u(t,W)$,  the optimal payoff of player $II$ can be defined from
 the backward Cauchy problem for the HJB equation
 \[
0=\frac{\pa S^{II}}{\pa t} +(\al(W), \nabla S^{II})+\De_{LB}S^{II}+\langle J^{II}\rangle_W
\]
\[
+\sum_{j,k,p,q}{Re} [i (w_{jk}A_{00,pq}-A_{jk,pq}) w_{pq}]\frac{\pa S^{II}}{\pa x_{jk}}
+\sum_{j,k,p,q}{Im} [i (w_{jk}A_{00,pq}-A_{jk,pq}) w_{pq}]\frac{\pa S^{II}}{\pa y_{jk}}
\]
\[
+ \left\{u \sum_{j,k,p}  \left[ {Re}[i w_{jk}h^I_{0p} w_{p0}-ih^I_{jp} w_{pk}]\frac{\pa S^{II}}{\pa x_{jk}}
+ {Im} [i w_{jk}h^I_{0p} w_{p0}-ih^I_{jp} w_{pk}] \frac{\pa S^{II}}{\pa y_{jk}}\right]\right\}
\]
 \begin{equation}
\label{eqHJBnonzero1}
+\sup_v\left\{v \sum_{j,k,p} \left[ {Re} [iw_{jk} h^{II}_{p0}w_{0p}-ih^{II}_{pk}w_{jp}]\frac{\pa S^{II}}{\pa x_{jk}}
+ {Im} [iw_{jk} h^{II}_{p0}w_{0p}-ih^{II}_{pk}w_{jp}] \frac{\pa S^{II}}{\pa y_{jk}}\right]\right\},
\end{equation}
where $\al$ includes the contributions arising from the uncontrolled
Hamiltonian operators (if any)  and from the 1st order terms
of the Laplace-Beltrami operator $\De_{LB} $ on the projective space $P\C^{2n+1}$.
Similarly for player I. Since $\sup_v$ and $\sup_u$ depend only on the signs ($\sgn$)  of the corresponding sums,
the pair of costs functions $S^I$ and $S^{II}$ satisfy the coupled system of two equations
(a vector-valued HJB):
 \[
0=\frac{\pa S^{I,II}}{\pa t} +(\al(W), \nabla S^{I,II})+\De_{LB}S^{I,II}+\langle J^{I,II}\rangle_W
\]
\[
+\sum_{j,k,p,q}{Re} [i (w_{jk}A_{00,pq}-A_{jk,pq}) w_{pq}]\frac{\pa S^{I,II}}{\pa x_{jk}}
+\sum_{j,k,p,q}{Im} [i (w_{jk}A_{00,pq}-A_{jk,pq}) w_{pq}]\frac{\pa S^{I,II}}{\pa y_{jk}}
\]
\[
+ \left\{u \sum_{j,k,p}  \left[ {Re}[i w_{jk}h^I_{0p} w_{p0}-ih^I_{jp} w_{pk}]\frac{\pa S^{I,II}}{\pa x_{jk}}
+ {Im} [i w_{jk}h^I_{0p} w_{p0}-ih^I_{jp} w_{pk}] \frac{\pa S^{I,II}}{\pa y_{jk}}\right]\right\}
\]
 \begin{equation}
\label{eqHJBnonzero2}
+\left\{v \sum_{j,k,p} \left[ {Re} [iw_{jk} h^{II}_{p0}w_{0p}-ih^{II}_{pk}w_{jp}]\frac{\pa S^{I,II}}{\pa x_{jk}}
+ {Im} [iw_{jk} h^{II}_{p0}w_{0p}-ih^{II}_{pk}w_{jp}] \frac{\pa S^{II}}{\pa y_{jk}}\right]\right\},
\end{equation}
with
 \begin{equation}
\label{eqHJBnonzero3}
\begin{aligned}
& u=U {\sgn} \left\{ \sum_{j,k,p}  \left[ {Re}[i w_{jk}h^I_{0p} w_{p0}-ih^I_{jp} w_{pk}]\frac{\pa S^{I,II}}{\pa x_{jk}}
+ {Im} [i w_{jk}h^I_{0p} w_{p0}-ih^I_{jp} w_{pk}] \frac{\pa S^{I,II}}{\pa y_{jk}}\right]\right\} \\
& v=V {\sgn} \left\{\sum_{j,k,p} \left[ {Re} [iw_{jk} h^{II}_{p0}w_{0p}-ih^{II}_{pk}w_{jp}]\frac{\pa S^{I,II}}{\pa x_{jk}}
+ {Im} [iw_{jk} h^{II}_{p0}w_{0p}-ih^{II}_{pk}w_{jp}] \frac{\pa S^{II}}{\pa y_{jk}}\right]\right\}.
\end{aligned}
\end{equation}

Since $u,v$ depend Lipschitz continuously on the gradients of $S^{I,II}$ and are uniformly bounded,
Theorem \ref{thHJBsmoothwel}
applies (more exactly, its straightforward vector-valued extension) leading to the well-posedness of system \eqref{eqHJBnonzero2}-\eqref{eqHJBnonzero3} in the sense of mild and/or classical solutions.
By the verification theorem (see e.g. \cite{FlemSon}; note that for checking the Nash condition one has to verify the optimality
for each single player, that is the verification theorem of the  standard control theory is applicable)
the solution of the backward Cauchy problem for system
\eqref{eqHJBnonzero2}-\eqref{eqHJBnonzero3}
yields the subgame-perfect Nash equilibrium for the corresponding game.
For the recent results on general non-zero-sum differential games we refer to \cite{GenerStochDiffGam219}
and references therein.

\section{Conclusions}
\label{conc}

We introduced the special homodyne detection schemes
that turn the problems of dynamic quantum filtering, control and games into the problems of the drift control
of the standard Brownian motions on the complex projective spaces, tori and Euclidean spaces
allowing for the effective theory of quantum dynamic games based on the classical and mild solutions
of the HJB-Isaacs equations on Riemannian manifolds. An explicitly solved example is presented.

This approach opens the road to the effective application of the recent advanced numeric approaches
to solving HJB equations, see \cite{MacPhys}, \cite{McEa09} and references therein, as well as to
the methods of finding explicit solutions from \cite{DuncTyr18}. Of course, additional work
is required for the concrete applications
of these methods to the present setting.

An interesting question arises from our construction.
What is the minimal number $N$ of the operators $L_j$ (physically, of optical measuring devices)
for a quantum system in $\C^{n+1}$
that can ensure that the resulting diffusion on $P\C^n$ is everywhere nondegenerate (and hence the theory of
Section \ref{secgameon Riemann} applies)? From Section \ref{qugametoprsp} it follows that $2n\le N \le n^2+2n$.
In particular $2\le N\le 3$ for a qubit. Notice that the scheme of Section \ref{secqugametoEuc} does not solve the problem
(as may be thought superficially), as it constructs the scheme with $N=2n$, which is nondegenerate everywhere
on the chart specified by finite $W$, but not outside it (in particular, with the exception of one point for a qubit).

Additionally, one can look at unbounded coefficients control problems arising from the homodyne schemes of Section
\ref{secqugametoEuc}.

Of interest is also a proper investigation of the long time behavior of controlled quantum processes,
which can lead to some kind of turnpike behavior (see \cite{KoWe12}) of stationary solutions. Some steps in this direction
were made in \cite{Kol92} and \cite{Kol95}, \cite{Barchonm} for the jump-type and diffusive filtering respectively. Some
application of these ideas can be found in \cite{Yurev}.

\end{document}